\documentclass[a4paper, 10pt]{article}
\usepackage{mathrsfs, amsmath, amsthm}
\usepackage{graphicx, color}
\usepackage{caption, subcaption}
\usepackage{array}
\usepackage{cases}
\makeatletter
\def\th@plain{%
  \upshape 
}
\makeatother

\makeatletter
\renewenvironment{proof}[1][\proofname]{\par
  \pushQED{\qed}%
  \normalfont \topsep6\p@\@plus6\p@\relax
  \trivlist
  \item[\hskip\labelsep
        \bfseries
    #1\@addpunct{.}]\ignorespaces
}{%
  \popQED\endtrivlist\@endpefalse
}
\makeatother

\newtheorem{theorem}{Theorem}
\numberwithin{theorem}{section}
\newtheorem{lemma}{Lemma}
\newtheorem{corollary}{Corollary}

\newtheorem*{conjecture*}{Conjecture}

\newtheorem{observation}{Observation}

\theoremstyle{definition}

\newtheorem{remark}{Remark}

\usepackage[left=25mm,top=25mm,bottom=25mm,right=25mm]{geometry}
\usepackage[T1]{fontenc}
\usepackage[varg]{txfonts}

\usepackage{enumitem}
\usepackage[square, numbers, sort&compress]{natbib}

\ifx\pdfoutput\undefined
 \usepackage[dvipdfm,%
  bookmarks=true,%
  bookmarksnumbered=true, 
  bookmarksopen=true, 
  plainpages=false,%
  pdfpagelabels,%
  colorlinks=true, 
  linkcolor=blue, 
  citecolor=blue,%
  hyperindex=true,
  urlcolor=black]{hyperref}
\else
 \usepackage[pdftex,%
  bookmarks=true,%
  bookmarksnumbered=true, 
  bookmarksopen=true, 
  plainpages=false,%
  pdfpagelabels,%
  colorlinks=true, 
  linkcolor=blue, 
  citecolor=blue,%
  anchorcolor=green,
  urlcolor= blue,
  breaklinks=true,
  hyperindex=true]{hyperref}
\fi

\newcounter{Hcase}
\newcounter{Hclaim}

\newcommand{\etal}{et~al.\ }
\newcommand{\ie}{i.e.,\ }

\def\int(#1){\mathrm{int}(#1)}
\def\ext(#1){\mathrm{ext}(#1)}
\def\Int(#1){\mathrm{Int}(#1)}
\def\Ext(#1){\mathrm{Ext}(#1)}
\def\ad(#1){\mathrm{ad}(#1)}
\def\mad(#1){\mathrm{mad}(#1)}
\def\la(#1){\mathrm{la}(#1)}

\DeclareMathOperator {\diam}{diam}
\DeclareMathOperator {\dist}{dist}
\DeclareMathOperator {\cor}{cor}

\begin{document}%
\title{On the diameter of total domination vertex critical graphs}
\author{Tao Wang\,\textsuperscript{a, b, }\footnote{{\tt Corresponding
author: wangtao@henu.edu.cn} }\\
{\small \textsuperscript{a}Institute of Applied Mathematics}\\
{\small Henan University, Kaifeng, 475004, P. R. China}\\
{\small \textsuperscript{b}College of Mathematics and Information Science}\\
{\small Henan University, Kaifeng, 475004, P. R. China}}
\date{August 5, 2015}
\maketitle
\begin{abstract}%
In this paper, we consider various types of domination vertex critical graphs, including total domination vertex critical graphs, independent domination vertex critical graphs and connected domination vertex critical graphs. We provide upper bounds on the diameter of them, two of which are sharp.


\end{abstract}
\section{Introduction}

All graphs considered here are finite, undirected and simple. Let $G$ be a graph with vertex set $V$ and edge set $E$. The {\em neighborhood} of a vertex $v$ in a graph $G$, denoted by $N_{G}(v)$, is the set of all the vertices adjacent to the vertex $v$, \ie $N_{G}(v) = \{u \in
V(G) \mid uv \in E(G)\}$, and the {\em closed neighborhood} of a vertex $v$ in $G$,
denoted by $N_{G}[v]$, is defined by $N_{G}[v] = N_{G}(v) \cup
\{v\}$. A vertex of degree one is called a
{\em leaf vertex}, the edge connected to that vertex is called a {\em pendant edge} and the only neighbor of a leaf vertex is called a {\em support vertex}. We denote the distance between $u$ and $v$ in $G$ by $\dist_{G}(u, v)$, and denote the diameter of $G$ by $\diam(G)$. The {\em degree} of a vertex $v$ in $G$, denoted by $\deg(v)$, is the
number of incident edges of $G$. A vertex of degree $k$ is called a $k$-vertex, and a vertex of degree at most or at least $k$ is call a $k^{-}$- or $k^{+}$-vertex, respectively.

A vertex subset $S \subseteq V$ is called a {\em dominating set} of a graph $G$ if every vertex in $V$ is an element of $S$ or is adjacent to a
vertex in $S$. The {\em domination number}
of a graph $G$, denoted by $\gammaup(G)$, is the minimum cardinality of a
dominating set of $G$. A graph is {\em domination vertex critical} if the removal of any vertex decreases its domination number. If $G$ is domination vertex critical and $\gammaup(G) = k$, we say that $G$ is a $k$-$\gammaup$-vertex-critical graph.

A vertex subset $S \subseteq V$ is a {\em total dominating set} of a graph $G$ if every vertex in $V$ is adjacent to
a vertex in $S$. Every graph without isolated vertices has a total dominating set, since $V$ is such a set. The {\em total domination number} of a graph $G$, denoted by $\gammaup_{t}(G)$, 
is the minimum cardinality of a total dominating set of $G$. A graph is {\em total domination vertex critical} if the removal of any vertex that is not adjacent to a vertex of degree one decreases its total domination number. If $G$ is total domination vertex critical and $\gammaup_{t}(G) = k$, we say that $G$ is a $k$-$\gammaup_{t}$-vertex-critical graph.

A vertex subset $S \subseteq V$ is an {\em independent dominating set} of a graph $G$ if it is a dominating set and it is also an independent set in $G$. Equivalently, an independent dominating set is a maximal independent set. The {\em independent domination number} of a graph $G$, denoted by $i(G)$,
is the minimum cardinality of an independent dominating set of $G$. A graph is {\em independent domination vertex critical} if the removal of any vertex decreases its independent domination number. If $G$ is independent domination vertex critical and $i(G) = k$, we say that $G$ is a $k$-$i$-vertex-critical graph.

A vertex subset $S \subseteq V$ is a {\em connected dominating set} of a graph $G$ if it is a dominating set of $G$ and the subgraph induced by $S$ is connected. Every connected graph has a connected dominating set, since $V$ is such a set. The {\em connected domination number} of a graph $G$, denoted by $\gammaup_{c}(G)$,
is the minimum cardinality of a connected dominating set of $G$. A graph is {\em connected domination vertex critical} if the removal of any vertex decreases its connected domination number. If $G$ is connected domination vertex critical and $\gammaup_{c}(G) = k$, we say that $G$ is a $k$-$\gammaup_{c}$-vertex-critical graph. A necessary condition for a graph to be $k$-$\gammaup_{c}$-vertex-critical is $2$-connected. For more details on connected domination vertex critical graphs, see \cite{MR2513321}.

The total domination vertex critical graphs were first investigated by Goddard \etal \cite{Goddard2004} and the independent domination vertex critical graphs were studied by Ao \cite{Ao1994}.

Goddard \etal \cite{Goddard2004} characterized the class of $k$-$\gammaup_{t}$-vertex-critical graphs with leaf vertices.
\begin{theorem}%
Let $G$ be a connected graph of order at least three with at least one leaf vertex. Then $G$ is $k$-$\gammaup_{t}$-vertex-critical if and only if $G = \cor(H)$ for some connected graph $H$ of order $k$ with $\delta(H) \geq 2$. 
\end{theorem}

For the connected $k$-$\gammaup_{t}$-vertex-critical graph without leaf vertices, they gave an upper bound on the diameter.
\begin{theorem}[Goddard \etal \cite{Goddard2004}]%
If $G$ is a connected $k$-$\gammaup_{t}$-vertex-critical graph without leaf vertices, then $\diam(G) \leq 2k-3$.
\end{theorem}

The following observation is used frequently, we present it here. 
\begin{observation}\label{obs}
If $D$ is a total dominating set of a graph $G$, then for every vertex $v$ in $G$, the set $D$ contains a neighbor of $v$.
\end{observation}

\begin{lemma}\label{Containing}%
If $G$ is a $k$-$\gammaup_{t}$-vertex-critical graph without leaf vertices, then for any vertex $w$, there exists a minimum total dominating set of $G$ containing $w$, and $\gammaup_{t}(G - w) = \gammaup_{t}(G) - 1$.
\end{lemma}
\begin{proof}%
Let $v$ be a neighbor of $w$ in $G$, and let $D$ be a minimum total dominating set of $G - v$. It follows that $w \notin D$ and $D \cap N_{G}(w) \neq \emptyset$, thus $D \cup \{w\}$ is a total dominating set of $G$. Furthermore, we have that $|D \cup \{w\}| = |D| + 1 \leq k$, and then $D \cup \{w\}$ is a minimum total dominating set of $G$ containing $w$ and $\gammaup_{t}(G - w) = |D| = \gammaup_{t}(G) -1$.
\end{proof}
\begin{lemma}%
If $G$ is a $k$-$i$-vertex-critical graph, then for any vertex $v$, there exists a minimum independent dominating set of $G$ containing $v$, and $i(G - v) = i(G) - 1$.
\end{lemma}

The method developed in \cite{MR1321108} is a powerful technique to obtain sharp upper bounds on various types of domination vertex critical graphs, it has been used on the $k$-$\gammaup$-vertex-critical graphs \cite{MR1321108} and paired domination vertex critical graphs \cite{MR2452766}.

Edwards and MacGillivray \cite{MR2917914} presented better upper bounds on the diameter of total domination and independent domination vertex critical graphs, but the proofs have big gaps (the gaps have been confirmed by Edwards via personal email). In this paper, we also adopt the technique in \cite{MR2917914, MR2452766, MR1321108} to obtain sharp upper bounds on the diameter, one of which is a slightly improvement on a result in \cite{MR2917914}. 

\section{Upper bounds on the diameter}

\begin{theorem}\label{TDIAM}
If $G$ is a connected $k$-$\gammaup_{t}$-vertex-critical graph without leaf vertices and $k \geq 4$, then $\diam(G) \leq \frac{5k-7}{3}$.
\end{theorem}
\begin{proof}%
Let $x$ and $x_{n}$ be vertices such that $\dist(x, x_{n}) = \diam(G) = n$. If $n \leq 4$, then we are done. So we may assume that $n \geq 5$. Let $xx_{1}\dots x_{n-1}x_{n}$ be a shortest path between $x$ and $x_{n}$. Define $L_{0}, L_{1}, \dots, L_{n}$ by $L_{i} = \{v \in V(G) \mid \dist_{G}(x, v) = i\}$ for $0 \leq i \leq n$. In particular, $L_{0} = \{x\}$ and $L_{1} = N_{G}(x)$. Let $R_{i} = L_{0} \cup L_{1} \cup \dots \cup L_{i}$ for $0 \leq i \leq n$. Let $D$ be a minimum total dominating set in $G$. If $|D \cap R_{j}| \geq \frac{3j + 10}{5}$, then we say that $R_{j}$ is {\em sufficient with respect to $D$}.

Let $m$ be the maximum integer $j$ such that $|D \cap R_{j}| \geq \frac{3j + 10}{5}$. Notice that the value of $m$ depends on the minimum total dominating set $D$, we may assume that $D$ is chosen such that $m$ is maximum among all the minimum total dominating set. 

Firstly, we must show the existence of $m$. Let $D_{1}$ be a minimum total dominating set of $G - x_{1}$. It is obvious that $x \notin D_{1}$ and $D_{1} \cap L_{1} \neq \emptyset$ and $|D_{1} \cap (L_{1} \cup L_{2})| \geq 2$. Suppose that the value of $m$ does not exist, it follows that $1 + |D_{1} \cap R_{j}| < \frac{3j + 10}{5}$, otherwise $R_{j}$ is sufficient with respect to $D_{1} \cup \{x\}$. Hence, we have that $|D_{1} \cap L_{1}| = 1$ and $|D_{1} \cap (L_{1} \cup L_{2})| < 2.2$. In fact $|D_{1} \cap L_{2}| = 1$. From the fact that $|D_{1} \cap R_{3}| < 2.8$, we have that $D_{1} \cap L_{3} = \emptyset$. If $D_{1} \cap L_{4} \neq \emptyset$, then we can conclude that $|D_{1} \cap (L_{4} \cup L_{5})| \geq 2$ from the fact that $D_{1}$ is a total dominating set of $G - x_{1}$, and then $R_{5}$ is sufficient with respect to $D_{1} \cup \{x\}$, a contradiction. So we may assume that $D_{1} \cap L_{4} = \emptyset$. Let $D_{0}$ be a minimum total dominating set of $G - x_{4}$. If $|D_{0} \cap R_{3}| \geq 3$, then $R_{3}$ is sufficient with respect to $D_{0} \cup \{x_{3}\}$, a contradiction. Thus, we have $|D_{0} \cap R_{3}| = 2$ and $D_{0} \cap L_{3} = \emptyset$. 
If $D_{0} \cap L_{4} = \emptyset$, then $D_{0} \cap R_{3}$ totally dominates $R_{3}$, and then $(D_{0} \cap R_{3}) \cup (D_{1} \setminus R_{3})$ is a smaller total dominating set of $G$, a contradiction. Hence we have that $D_{0} \cap L_{4} \neq \emptyset$ and $|D_{0} \cap (L_{4} \cup L_{5})| \geq 2$. Therefore, we have that $|D_{0} \cap R_{5}| \geq 4$ and the set $R_{5}$ is sufficient with respect to $D_{0} \cup \{x_{3}\}$, which leads to a contradiction.

Now, we know that the value of $m$ must exist. If $m = n$, then $n = m \leq \frac{5k-10}{3} \leq \frac{5k-7}{3}$, we are done. So we may assume that $m < n$.

If $m = 5t + 2$, then $|D \cap R_{m}| \geq 3t + 3.2$ and $|D \cap R_{m+1}| < 3t + 3.8$, which is a contradiction. If $m = 5t + 4$, then $|D \cap R_{m}| \geq 3t + 4.4$ and $|D \cap R_{m+1}| < 3t + 5$, which is also a contradiction. So we have that $m = 5t, 5t+1$ or $5t+3$. We further assume that $m + 2 \leq n$. 

If $m = 5t$, then $|D \cap R_{m}| \geq 3t + 2$ and $|D \cap R_{m+1}| < 3t + 2.6$, which implies that $|D \cap R_{m}| = 3t+2$ and $D \cap L_{m+1} = \emptyset$. From the fact that $|D \cap R_{m+2}| < 3t + 3.2$ and $D$ is a total dominating set and $|D \cap R_{m+3}| < 3t + 3.8$ (if $L_{m+3}$ exists), we can conclude that $D \cap L_{m+2} = \emptyset$ and $|D \cap L_{m+3}| = 1$. Consequently, the set $L_{m+4}$ exists and $D \cap L_{m+3}$ dominates $L_{m+2}$. But $|D \cap R_{m+4}| < 3t + 4.4$, so we have that $|D \cap L_{m+4}| = 1$.

If $m = 5t + 1$, then $|D \cap R_{m}| \geq 3t + 2.6$, $|D \cap R_{m+1}| < 3t + 3.2$ and $|D \cap R_{m+2}| < 3t + 3.8$, which implies that $|D \cap R_{m}| = 3t+3$ and $D \cap L_{m+1} = D \cap L_{m+2} = \emptyset$. In order to dominate $L_{m+2}$, the set $L_{m+3}$ exists and $D \cap L_{m+3}$ dominates $L_{m+2}$. But $|D \cap R_{m+3}| < 3t + 4.4$, so we have that $|D \cap L_{m+3}| = 1$. The set $D$ totally dominates $G$, it follows that $L_{m+4}$ exists and $D \cap L_{m+4} \neq \emptyset$. Hence $|D \cap R_{m+4}| \geq 3t + 5$ and $R_{m+4}$ is sufficient with respect to $D$, a contradiction to the maximality of $m$.

If $m = 5t + 3$, then $|D \cap R_{m}| \geq 3t + 3.8$, $|D \cap R_{m+1}| < 3t + 4.4$ and $|D \cap R_{m+2}| < 3t + 5$, which implies that $|D \cap R_{m}| = 3t+4$ and $D \cap L_{m+1} = D \cap L_{m+2} = \emptyset$. In order to dominate $L_{m+2}$, the set $L_{m+3}$ exists and $D \cap L_{m+3}$ dominates $L_{m+2}$. But $|D \cap R_{m+3}| < 3t + 5.6$, so we have that $|D \cap R_{m+3}| = 1$. Since $D$ is a total dominating set in $G$, it follows that $L_{m+4}$ exists and $D \cap L_{m+4} \neq \emptyset$, but with $|D \cap R_{m+4}| < 3t + 6.2$, we have that $|D \cap L_{m+4}| = 1$.

By the above arguments, we may assume that $D \cap L_{m+1} = D \cap L_{m+2} = \emptyset$ and $|D \cap L_{m+3}| = |D \cap L_{m+4}| = 1$, where $m = 5t$ or $m = 5t+3$. Without loss of generality, we assume that $D \cap L_{m+3} = \{x_{m+3}\}$ and $D \cap L_{m+4} = \{x_{m+4}\}$. Let $D_{3}$ and $D_{4}$ be a minimum total dominating set of $G - x_{m+3}$ and $G - x_{m+4}$, respectively.

Recall that the vertex $x_{m+3}$ dominates $L_{m+2}$, then $D_{3} \cap L_{m+2} = \emptyset$ and $D_{3} \cap R_{m+1}$ totally dominates $R_{m+1}$. If $|D_{3} \cap R_{m+1}| < |D \cap R_{m+1}|$, then $(D_{3} \cap R_{m+1}) \cup (D \setminus R_{m+1})$ is a smaller total dominating set in $G$, which leads to a contradiction. If $|D_{3} \cap R_{m+1}| > |D \cap R_{m+1}|$, then $R_{m+1}$ is sufficient with respect to the minimum total dominating set $D_{3} \cup \{x_{m+4}\}$. Hence we have that $|D_{3} \cap R_{m+1}| = |D \cap R_{m+1}|$. Notice that maybe $L_{m+5}$ does not exist, if this happens, then we view $L_{m+5}$ as an empty set. If $|D_{3} \cap (L_{m+3} \cup L_{m+4} \cup L_{m+5})| \geq 2$, then $|(D_{3} \cup \{x_{m+4}\}) \cap R_{m+5}| \geq |D \cap R_{m+1}| + 3$, and then $R_{m+5}$ (or $R_{m+4}$ if $L_{m+5}$ does not exist) is sufficient with respect to $D_{3} \cup \{x_{m+4}\}$, which contradicts the maximality of $m$. Hence, we have that $|D_{3} \cap (L_{m+3} \cup L_{m+4} \cup L_{m+5})| \leq 1$, which implies that $L_{m+5}$ exists and $D_{3} \cap L_{m+4} = \emptyset$ and $L_{m+3} = \{x_{m+3}\}$.

Notice that $D_{4} \cap L_{m+3} = \emptyset$ and $D_{4} \cap R_{m+2}$ totally dominates $R_{m+2}$. If $|D_{4} \cap R_{m+2}| < |D \cap R_{m+2}|$, then $(D_{4} \cap R_{m+2}) \cup (D \setminus R_{m+2})$ is a smaller total dominating set of $G$, which leads to a contradiction. If $|D_{4} \cap R_{m+2}| > |D \cap R_{m+2}|$, then $|(D_{4} \cup \{x_{m+3}\}) \cap R_{m+3}| \geq |D \cap R_{m}| + 2$, and then $R_{m+3}$ is sufficient with respective to $D_{4} \cup \{x_{m+3}\}$, which leads to a contradiction. Hence, we have that $|D_{4} \cap R_{m+2}| = |D \cap R_{m+2}|$. 

If $D_{4} \cap L_{m+2} \neq \emptyset$, then  $(D_{4} \cap R_{m+2}) \cup (D_{3} \setminus R_{m+2})$ is a smaller total dominating set of $G$, a contradiction. It follows that $D_{4} \cap L_{m+2} = \emptyset$. In order to dominate the vertex $x_{m+3}$, we must have that $D_{4} \cap L_{m+4} \neq \emptyset$. Hence, we can conclude that $|D_{4} \cap (L_{m+4} \cup L_{m+5})| \geq 2$ and $R_{m+5}$ is sufficient with respect to $D_{4} \cup \{x_{m+3}\}$, a contradiction.

Finally, we have to deal with the case that $m = n - 1$. Recall that $m$ is the maximum integer $j$ such that $|D \cap R_{j}| \geq \frac{3j + 10}{5}$, it follows that $D \cap L_{m+1} = D \cap L_{n} = \emptyset$, and then $|D \cap R_{m}| = k$ and $n = m + 1 \leq \frac{5k-10}{3} + 1 = \frac{5k-7}{3}$.
\end{proof}

The {\em coalescence} of two graphs $G_{1}$ and $G_{2}$ with respect to a vertex $x$ in $G_{1}$ and a vertex $y$ in $G_{2}$, is the graph $G_{1}(x*y)G_{2}$ obtained by identifying $x$ and $y$; in other words, replacing the vertices $x$ and $y$ by a new vertex $w$ adjacent to the same vertices in $G_{1}$ as $x$ and the same vertices in $G_{2}$ as $y$. If there is no confusion, then we write $G_{1} * G_{2}$ instead of $G_{1}(x*y)G_{2}$.

\begin{theorem}\label{IDIAM}%
If $G$ is a connected $k$-$i$-vertex-critical graph, then $\diam(G) \leq 2(k-1)$.
\end{theorem}
\begin{proof}%
Let $x$ and $x_{n}$ be vertices such that $\dist(x, x_{n}) = \diam(G) = n$. Let $xx_{1}\dots x_{n-1}x_{n}$ be a shortest path between $x$ and $x_{n}$. Define $L_{0}, L_{1}, \dots, L_{n}$ by $L_{i} = \{v \in V(G) \mid \dist_{G}(x, v) = i\}$ for $0 \leq i \leq n$. In particular, $L_{0} = \{x\}$ and $L_{1} = N_{G}(x)$. Let $R_{i} = L_{0} \cup L_{1} \cup \dots \cup L_{i}$ for $0 \leq i \leq n$. Let $D$ be a minimum independent dominating set in $G$. If $|D \cap R_{j}| \geq \frac{j + 2}{2}$, then we say that $R_{j}$ is {\em sufficient with respect to $D$}.

Let $m$ be the maximum integer $j$ such that $|D \cap R_{j}| \geq \frac{j + 2}{2}$. The value of $m$ depends on the minimum independent dominating set $D$, we may assume that $D$ is chosen such that $m$ is maximum among all the minimum independent dominating set. Let $D_{1}$ be a minimum independent dominating set of $G - x_{1}$. It is obvious that $x \notin D_{1}$ and $D_{1} \cap L_{1} \neq \emptyset$. Thus $D_{1} \cup \{x_{1}\}$ is a minimum independent dominating set of $G$ with $|(D_{1} \cup \{x_{1}\}) \cap R_{1}| \geq 2$, and then the value of $m$ exists and $m \geq 1$. If $m = n$, then $n = m \leq 2(k-1)$, we are done. So we may assume that $m < n$.

If $m = 2t + 1$, then $|D \cap R_{m}| \geq t + 1.5$ and $|D \cap R_{m+1}| < t + 2$, which is a contradiction. So we have that $m = 2t$. We further assume that $m + 2 \leq n$. It follows that $|D \cap R_{m}| \geq t + 1$ and $|D \cap R_{m+1}| < t + 1.5$ and $|D \cap R_{m+2}| < t + 2$, and then $|D \cap R_{m}| = t + 1$ and $D \cap L_{m+1} = D \cap L_{m+2} = \emptyset$. In order to dominate $L_{m+2}$, the set $L_{m+3}$ must exist and $D \cap L_{m+3}$ dominates $L_{m+2}$. The fact that $|D \cap R_{m+3}| < t + 2.5$ implies that $|D \cap L_{m+3}| = 1$. Let $D \cap L_{m+3} = \{w\}$. Notice that if $L_{m+4}$ exists, we can conclude that $D \cap L_{m+4} = \emptyset$ from the fact that $|D \cap R_{m+4}| < t + 3$. Hence, the vertex $w$ dominates $L_{m+2} \cup L_{m+3}$.

Let $D_{3}$ be a minimum independent dominating set of $G - w$. Notice that $D_{3} \cap (L_{m+2} \cup L_{m+3}) = \emptyset$. If $|D_{3} \cap R_{m+1}| > |D \cap R_{m+1}|$, then $R_{m+1}$ is sufficient with respect to $D_{3} \cup \{w\}$. If $|D_{3} \cap R_{m+1}| < |D \cap R_{m+1}|$, then $(D_{3} \cap R_{m+1}) \cup (D \setminus R_{m+1})$ is a smaller independent dominating set of $G$, which is a contradiction. Hence, we have that $|D_{3} \cap R_{m+1}| = |D \cap R_{m+1}|$.

Suppose that the set $L_{m+4}$ does not exist. It implies that $|D \cap R_{m}| = k - 1 = t+1$. Recall that $w$ dominates $L_{m+2} \cup L_{m+3}$, it follows that $D_{3} \subseteq R_{m+1}$ and $L_{m+3} = \{w\}$. Let $D_{2}$ be a minimum independent dominating set of $G - x_{m+2}$. Therefore, the set $D_{2} \cup \{x_{m+2}\}$ is a minimum independent dominating set with $|(D_{2} \cup \{x_{m+2}\}) \cap R_{m+2}| =  k = t+2$, thus $R_{m+2}$ is sufficient with respect to $D_{2} \cup \{x_{m+2}\}$, which is a contradiction. So we may assume that $L_{m+4}$ exists.

If $|D_{3} \cap (L_{m+3} \cup L_{m+4})| \geq 1$, then $R_{m+4}$ is sufficient with respect to $D_{3} \cup \{w\}$, which leads to a contradiction. So we have that $D_{3} \cap (L_{m+3} \cup L_{m+4}) =  \emptyset$ and $L_{m+3} = \{w\}$. Let $D_{4}$ be a minimum independent dominating set of $G - x_{m+4}$. Notice that $D_{4} \cap L_{m+3} = \emptyset$ and $D_{4} \cap R_{m+2}$ totally dominates $R_{m+2}$. If $|D_{4} \cap R_{m+2}| > |D \cap R_{m+2}|$, then $R_{m+2}$ is sufficient with respect to $D_{4} \cup \{x_{m+4}\}$. 

If $|D_{4} \cap R_{m+2}| \leq |D \cap R_{m+2}|$ and $D_{4} \cap L_{m+2} \neq \emptyset$, then $(D_{4} \cap R_{m+2}) \cup (D_{3} \setminus R_{m+3})$ is a smaller independent dominating set of $G$, a contradiction. 

If $|D_{4} \cap R_{m+2}| = |D \cap R_{m+2}|$ and $D_{4} \cap L_{m+2} = \emptyset$, then $D_{4} \cap L_{m+4} \neq \emptyset$ in order to dominates $w$, and then $R_{m+4}$ is sufficient with respect to $D_{4} \cup \{x_{m+4}\}$.

If $|D_{4} \cap R_{m+2}| < |D \cap R_{m+2}|$ and $D_{4} \cap L_{m+2} = \emptyset$, then $(D_{4} \cap R_{m+2}) \cup (D \setminus R_{m+2})$ is a smaller independent dominating set of $G$, a contradiction.

By the above arguments, the theorem is true except the case that $m = 2t = n - 1$. Notice that $G (x*x) G$ is a $(2k-1)$-$i$-vertex-critical graph with diameter $2n$. The theorem is true for the graph $G (x*x) G$, it implies that $2n \leq 2(2k-1 - 1)$, thus $n \leq 2(k-1)$.
\end{proof}

\begin{theorem}%
If $G$ is a $k$-$\gammaup_{c}$-vertex-critical graph, then $\diam(G) \leq k$.
\end{theorem}
\begin{proof}%
Let $x$ and $x_{n}$ be vertices such that $\dist(x, x_{n}) = \diam(G) = n$. Let $xx_{1}\dots x_{n-1}x_{n}$ be a shortest path between $x$ and $x_{n}$. Define $L_{0}, L_{1}, \dots, L_{n}$ by $L_{i} = \{v \in V(G) \mid \dist_{G}(x, v) = i\}$ for $0 \leq i \leq n$. In particular, $L_{0} = \{x\}$ and $L_{1} = N_{G}(x)$. Let $D_{1}$ be a minimum connected dominating set of $G - x_{1}$. It is obviously that $x \notin D_{1}$ and $D_{1} \cap L_{1} \neq \emptyset$. Since $D_{1}$ is a connected dominating set of $G$, it follows that $D_{1} \cap L_{i} \neq \emptyset$ for every $1 \leq i \leq n-1$. Hence we have that $|D_{1}| = k - 1 \geq n - 1$, which implies that $\diam(G) = n \leq k$.
\end{proof}

\section{Sharpness of the upper bounds}

We characterize when the coalescence of two total domination vertex critical graphs is still a total domination vertex graph.

\begin{theorem}\label{coalescence}%
Let $G_{1}$ and $G_{2}$ be $k_{1}$-$\gammaup_{t}$-vertex-critical and $k_{2}$-$\gammaup_{t}$-vertex-critical graphs without leaf vertices, respectively. Let $x$ and $y$ be two vertices in $G_{1}$ and $G_{2}$, respectively. Then $G_{1}(x*y)G_{2}$ is $(k_{1} + k_{2} -1)$-$\gammaup_{t}$-vertex-critical if and only if $\gammaup_{t}(G_{2} - N_{G_{2}}[y]) \geq k_{2} -1$ and $\gammaup_{t}(G_{1} - N_{G_{1}}[x]) \geq k_{1} -1$.
\end{theorem}
\begin{proof}%
Denote the graph $G_{1}(x*y) G_{2}$ by $G$ for short. Let $D$ be a minimum total dominating set of $G$ and $w$ be the new created vertex in $G$. Let $D_{1}$ and $D_{2}$ be a minimum total dominating set of $G_{1} - x$ and $G_{2} - y$, respectively. Thus $|D_{1}| = k_{1} - 1$ and $|D_{2}| = k_{2} - 1$. It is obvious that $\gammaup_{t}(G - w) = k_{1} + k_{2} - 2$. For any vertex $v \in V(G_{1}) \setminus \{x\}$, the union of $D_{2}$ and a minimum total dominating set of $G_{1} - v$ is a total dominating set of $G - v$, thus $\gammaup_{t}(G - v) \leq k_{1} + k_{2} - 2$. Similarly, for any vertex $v \in V(G_{2}) \setminus \{y\}$, the union of $D_{1}$ and a minimum total dominating set of $G_{2} - v$ is a total dominating set of $G - v$, and then $\gammaup_{t}(G - v) \leq k_{1} + k_{2} - 2$. Hence, for any vertex $v$ in $V(G)$, we have that $\gammaup_{t}(G - v) \leq k_{1} + k_{2} - 2$.

($\Longleftarrow$) Suppose that $\gammaup_{t}(G_{2} - N_{G_{2}}[y]) \geq k_{2} -1$ and $\gammaup_{t}(G_{1} - N_{G_{1}}[x]) \geq k_{1} -1$. We want to prove $\gammaup_{t}(G) \geq k_{1} + k_{2} - 1$.

Notice that either $D \cap V(G_{1})$ totally dominates $G_{1}$ or $D \cap V(G_{2})$ totally dominates $G_{2}$. By symmetry, we may assume that $D \cap V(G_{1})$ totally dominates $G_{1}$ and $|D \cap V(G_{1})| \geq k_{1}$. If $w \notin D$, then $D \cap V(G_{2})$ totally dominates $G_{2} - y$ and $|D \cap V(G_{2})| \geq k_{2} -1$, and then $|D| \geq k_{1} + k_{2} - 1$. So we may assume that $w \in D$. If $D \cap N_{G_{2}}(y) \neq \emptyset$, then $D \cap V(G_{2})$ is a total dominating set of $G_{2}$ and $|D \cap V(G_{2})| \geq k_{2}$, and then $|D| \geq k_{1} + k_{2} - 1$. If $D \cap N_{G_{2}}(y) = \emptyset$, then $D \setminus V(G_{1}) \subseteq V(G_{2}) \setminus N_{G_{2}}[y]$ and $D \setminus V(G_{1})$ totally dominates $G_{2} - N_{G_{2}}[y]$, and then $|D \setminus V(G_{1})| \geq k_{2} - 1$ and $|D| \geq k_{1} + k_{2} - 1$.

($\Longrightarrow$) Suppose that $|D| = k_{1} + k_{2} -1$. We want to prove that $\gammaup_{t}(G_{2} - N_{G_{2}}[y]) \geq k_{2} -1$ and $\gammaup_{t}(G_{1} - N_{G_{1}}[x]) \geq k_{1} -1$.

By \autoref{Containing}, let $D_{1}^{*}$ be a minimum total dominating set of $G_{1}$ containing $x$. It follows that $\gammaup_{t}(G_{2} - N_{G_{2}}[y]) \geq k_{2} - 1$; otherwise, the union of $D_{1}^{*}$ and a minimum total dominating set of $G_{2} - N_{G_{2}}[y]$ is a smaller total dominating set of $G$, a contradiction. Similarly, we can prove that $\gammaup_{t}(G_{1} - N_{G_{1}}[x]) \geq k_{1} -1$.
\end{proof}
\begin{remark}%
From the characterization, the graph $C_{6} * C_{6}$ is not a total domination vertex critical graph as mentioned in \cite{MR2917914}.
\end{remark}

A {\em pointed graph} is a graph with two assigned diametrical vertices called \textsc{Left} and \textsc{Right}. For a pointed graph $G$, we define $L_{k}(G)$ and $R_{k}(G)$ be the set of vertices which are distance $k$ from the \textsc{Left}-vertex and \textsc{Right}-vertex, respectively.

For two pointed graphs $G_{1}$ and $G_{2}$, we define $G_{1} \bullet G_{2}$ as the pointed graph obtained by identifying and unassigning the \textsc{Right}-vertex from $G_{1}$ and the \textsc{Left}-vertex from $G_{2}$.

Let $K_{m, m}$ be a complete bipartite graph with bipartition $\{y_{1}, y_{3}, \dots, y_{2m-1}\}$ and $\{y_{2}, y_{4}, \dots, y_{2m}\}$, where $m \geq 2$. Let $F$ be the graph obtained from $K_{m, m}$ by removing one edge $y_{1}y_{2m}$, and let $\bar{F}$ be the complement of $F$ with $x_{i}$ corresponding to $y_{i}$. Notice that $\gammaup_{t}(F) = \gammaup_{t}(\bar{F}) = 2$ and $\{x_{1}, x_{2m}\}$ totally dominates $\bar{F}$ and every pair of adjacent vertices in $K_{m, m}$ totally dominates $K_{m, m}$. Let $R$ be the pointed graph obtained from the disjoint union of $\bar{F}$ and $K_{m, m}$, by joining every vertex of $\bar{F}$ to every vertex of $K_{m, m}$ except edges between the corresponding vertices, and adding five new vertices $z_{1}, z_{2}, z_{3}$, \textsc{Left} and \textsc{Right} such that \textsc{Left} is adjacent to every vertex in $\bar{F}$, the vertex $z_{1}$ is adjacent to $\{x_{1}, x_{2}, \dots, x_{2m-1}\} \cup \{y_{2}, y_{3}, \dots, y_{2m-1}\}$, the vertex $z_{2}$ is adjacent to $\{x_{2}, x_{3}, \dots, x_{2m}\} \cup \{y_{2}, y_{3}, \dots, y_{2m-1}\}$, the vertex $z_{3}$ is adjacent to every vertex in $K_{m, m}$ and $z_{1}$, while \textsc{Right} is adjacent to every vertex in $K_{m, m}$ and $z_{2}$.

\begin{theorem}%
The graph $R$ is $3$-$\gammaup_{t}$-vertex-critical graph with diameter three.
\end{theorem}

Let $H$ be a graph with at least four vertices. Let $V(H) = \{x_{1}, \dots, x_{t}\}$ and $V(\bar{H}) = \{y_{1}, \dots, y_{t}\}$ with $x_{i}$ corresponding to $y_{i}$. Let $A$ be the pointed graph obtained by joining every vertex of $H$ to every vertex of $\bar{H}$ except edges between the corresponding vertices, and adding two new vertices \textsc{Left} and \textsc{Right} such that \textsc{Left} is adjacent to every vertex in $H$ and \textsc{Right} is adjacent to every vertex in $\bar{H}$. It can be shown that $A$ is a $3$-$\gammaup_{t}$-vertex-critical graph if and only if $\gammaup_{t}(H) = \gammaup_{t}(\bar{H}) = 2$. Simply write the \textsc{Left}-vertex as $x$ and \textsc{Right}-vertex as $y$. Suppose that $\gammaup_{t}(H) = \gammaup_{t}(\bar{H}) = 2$. A minimum total dominating set of $H$ totally dominates $A - y$ and a minimum total dominating set of $\bar{H}$ totally dominates $A - x$. For any vertex $x_{i}$, the two vertices $y_{i}$ and a nonadjacent vertex $x_{j}$ of $x_{i}$ totally dominates $A - x_{i}$; similarly, for any vertex $y_{i}$, the two vertices $x_{i}$ and a nonadjacent vertex $y_{j}$ of $y_{i}$ totally dominates $A - y_{i}$. But $\gammaup_{t}(A) > 2$, thus $A$ is a $3$-$\gammaup_{t}$-vertex-critical graph. Conversely, if $G$ is a $3$-$\gammaup_{t}$-vertex-critical graph, then a minimum total dominating set of $A - y$ is also a minimum total dominating set of $H$ and a minimum total dominating set of $A - x$ is also a minimum total dominating set of $\bar{H}$, and then $\gammaup_{t}(H) = \gammaup_{t}(\bar{H}) = 2$. In what follows, we assume that $\gammaup_{t}(H) = \gammaup_{t}(\bar{H}) = 2$.  Notice that $\diam(A) = 3$.
 
 \begin{remark}%
For every $t \geq 4$, we can find at least one graph $H$ on $t$ vertices with $\gammaup_{t}(H) = 2$ and $\gammaup_{t}(\bar{H}) = 2$. For instance, let $K_{t-2}$ be a complete graph on $t-2$ vertices, and let $H$ be the graph on $t$ vertices obtained from $K_{t-2}$ by attaching a path $xx_{1}x_{2}$. It is easy to check that $\gammaup_{t}(H) = 2$ and $\gammaup_{t}(\bar{H}) = 2$.
\end{remark}

Let $Q$ be the pointed graph obtained from two copies of $A$, called $A_{1}$ and $A_{2}$, by deleting the \textsc{Right}-vertex $y$ from $A_{1}$ and the \textsc{Left}-vertex $x$ from $A_{2}$, and joining every neighbor of $y$ in $A_{1}$ to every neighbor of $x$ in $A_{2}$. Notice that $\diam(Q) = 5$ and $\gammaup_{t}(Q) = 4$. By \autoref{TDIAM}, the graph $Q$ is not a $4$-$\gammaup_{t}$-vertex-critical graph. Let $Q^{(1)} = Q$ and $Q^{(n)} = Q^{(n-1)} \bullet Q$. We simple denote $R  \bullet  Q^{(n)}$ by $\mathfrak{C}_{n}$.

Let $J_{1}$ and $J_{3}$ be disjoint union of $tK_{2}$ and let $J_{2}$ be $\bar{tK_{2}}$, where $t \geq 2$. Let $J$ be the pointed graph obtained from $J_{1} \cup J_{2} \cup J_{3}$ by joining every vertex of $J_{1}$ to every vertex of $J_{2}$ except the edges corresponding vertices in $J_{1}$ and $J_{2}$, similarly, joining every vertex of $J_{2}$ to every vertex of $J_{3}$ except the edges corresponding vertices in $J_{2}$ and $J_{3}$, adding a new \textsc{Left} vertex $x$ adjacent to every vertex of $J_{1}$ and adding a new \textsc{Right} vertex $y$ adjacent to every vertex in $J_{3}$. It is easy to check that $J$ is a $4$-$\gammaup_{t}$-vertex-critical graph with diameter $4$.

\begin{theorem}
(a) $\gammaup_{t}(R  \bullet  Q^{(n)}) \geq 3n + 3$; (b) $\gammaup_{t}(R  \bullet  Q^{(n)} - y) \geq 3n + 2$; (c) $\gammaup_{t}(R  \bullet  Q^{(n)} - N[y]) \geq 3n + 2$.
\end{theorem}
\begin{proof}%
We prove the results by mathematical induction.

{\bf Basis step:} If $n = 0$, then the results are trivially true.

{\bf Inductive step:} Suppose that the results are true for all values less than $n$. Let $D, D_{1}$ and $D_{2}$ be a minimum total dominating set of $\mathfrak{C}_{n}$, $\mathfrak{C}_{n} - y$ and $\mathfrak{C}_{n} - N[y]$, respectively. Denote the \textsc{Left} vertex of $Q_{n}$ by $x$ and the \textsc{Right} vertex of $Q_{n}$ by $y$. If $D \cap V(\mathfrak{C}_{n-1})$ totally dominates $\mathfrak{C}_{n-1}$, then $|D \cap V(\mathfrak{C}_{n-1})| \geq 3n$, but $|D \setminus V(\mathfrak{C}_{n-1})| \geq 3$, thus $|D| \geq 3n + 3$. So we may assume that $D \cap V(\mathfrak{C}_{n-1})$ does not totally dominates $\mathfrak{C}_{n-1}$. Notice that $D \cap V(Q_{n})$ must totally dominate $Q_{n}$ and $|D \cap V(Q_{n})| \geq 4$. If $x \notin D$, then $D \cap R_{1}(Q_{n-1}) = \emptyset$ and $D \cap V(\mathfrak{C}_{n-1})$ totally dominates $\mathfrak{C}_{n-1} - x$ and $|D \cap V(\mathfrak{C}_{n-1})| \geq 3n-1$, thus $|D| \geq 3n -1 + 4 = 3n+3$. If $x \in D$, then $D \cap R_{1}(Q_{n-1}) = \emptyset$ and $D \cap V(\mathfrak{C}_{n-1} -x)$ totally dominates $\mathfrak{C}_{n-1} - N[x]$ and $|D \cap V(\mathfrak{C}_{n-1} -x)| \geq 3n-1$. Since $D \cap V(Q_{n})$ totally dominates $Q_{n}$ and $x \in D$, it follows that $|D \cap V(Q_{n})| \geq 5$, and then $|D| \geq 3n - 1 + 5 = 3n+4$. Hence, we have that $\gammaup_{t}(R  \bullet  Q^{(n)}) \geq 3n + 3$.

If $D_{1} \cap V(\mathfrak{C}_{n-1})$ totally dominates $\mathfrak{C}_{n-1}$, then $|D_{1} \cap V(\mathfrak{C}_{n-1})| \geq 3n$, but $|D_{1} \setminus V(\mathfrak{C}_{n-1})| \geq 2$, thus $|D_{1}| \geq 3n + 2$. So we may assume that $D_{1} \cap V(\mathfrak{C}_{n-1})$ does not totally dominates $\mathfrak{C}_{n-1}$. If $x \notin D_{1}$, then $D_{1} \cap R_{1}(Q_{n-1}) = \emptyset$ and $D_{1} \cap V(\mathfrak{C}_{n-1})$ totally dominates $\mathfrak{C}_{n-1} - x$ and $|D_{1} \cap V(\mathfrak{C}_{n-1})| \geq 3n-1$. Notice that $D_{1} \setminus V(\mathfrak{C}_{n-1})$ totally dominates $Q_{n} - y$ and $D_{1} \cap L_{1}(Q_{n}) \neq \emptyset$ and $|D_{1} \setminus V(\mathfrak{C}_{n-1})| \geq 4$, thus $|D_{1}| \geq 3n - 1 + 4 = 3n + 3$. If $x \in D_{1}$, then $D_{1} \cap R_{1}(Q_{n-1}) = \emptyset$ and $D_{1} \cap V(\mathfrak{C}_{n-1} - x)$ totally dominates $\mathfrak{C}_{n-1} - N[x]$ and $|D_{1} \cap V(\mathfrak{C}_{n-1}-N[x])| \geq 3n - 1$. Notice that $x \in D_{1}$ and $D_{1} \cap L_{1}(Q_{n}) \neq \emptyset$, and  then $|D_{1} \cap (Q_{n} - y)| \geq 4$, thus $|D_{1}| \geq 3n - 1 + 4 = 3n + 3$. Hence, we have that $\gammaup_{t}(R  \bullet  Q^{(n)} - y) \geq 3n + 2$.

If $D_{2} \cap V(\mathfrak{C}_{n-1})$ totally dominates $\mathfrak{C}_{n-1}$, then $|D_{2} \cap V(\mathfrak{C}_{n-1})| \geq 3n$, but $|D_{2} \setminus V(\mathfrak{C}_{n-1})| \geq 2$, thus $|D_{2}| \geq 3n + 2$. So we may assume that $D_{2} \cap V(\mathfrak{C}_{n-1})$ does not totally dominates $\mathfrak{C}_{n-1}$. If $x \notin D_{2}$, then $D_{2} \cap R_{1}(Q_{n-1}) = \emptyset$ and $D_{2} \cap V(\mathfrak{C}_{n-1})$ totally dominates $\mathfrak{C}_{n-1} - x$ and $|D_{2} \cap V(\mathfrak{C}_{n-1})| \geq 3n-1$. Notice that $D_{2} \setminus V(\mathfrak{C}_{n-1})$ totally dominates $Q_{n} - N[y]$ and $D_{2} \cap L_{1}(Q_{n}) \neq \emptyset$ and $|D_{2} \setminus V(\mathfrak{C}_{n-1})| \geq 3$, thus $|D_{2}| \geq 3n - 1 + 3 = 3n + 2$. If $x \in D_{2}$, then $D_{2} \cap R_{1}(Q_{n-1}) = \emptyset$ and $D_{2} \cap V(\mathfrak{C}_{n-1} - x)$ totally dominates $\mathfrak{C}_{n-1} - N[x]$ and $|D_{2} \cap V(\mathfrak{C}_{n-1}-N[x])| \geq 3n - 1$. Notice that $x \in D_{2}$ and $D_{2} \cap L_{1}(Q_{n}) \neq \emptyset$, and then $|D_{2} \cap (Q_{n} - N[y])| \geq 3$, thus $|D_{1}| \geq 3n -1 + 3 = 3n + 2$. Hence, we have that $\gammaup_{t}(R  \bullet  Q^{(n)} - N[y]) \geq 3n + 2$.
\end{proof}

\begin{corollary}%
(a) $\gammaup_{t}(R  \bullet  Q^{(n)}) = 3n + 3$; (b) $\gammaup_{t}(R  \bullet  Q^{(n)} - y) = 3n + 2$; (c) $\gammaup_{t}(R  \bullet  Q^{(n)} - N[y]) = 3n + 2$.
\end{corollary}

\begin{theorem}%
The graph $R  \bullet  Q^{(n)}  \bullet  J$ is $(3n+6)$-$\gammaup_{t}$-vertex-critical graph with diameter $5n + 7$.
\end{theorem}
\begin{proof}%
If $n = 0$, then the statement follows by \autoref{coalescence}. So we may assume that $n \geq 1$. Denote the graph $R \bullet Q^{(n)} \bullet J$ by $G$ and denote the i-th copy of $Q$ by $Q_{i}$ with \textsc{Left} $x_{i}$ and \textsc{Right} $y_{i}$. Denote the \textsc{Left} vertex of $J$ by $x$ and the \textsc{Right} vertex by $y$. Let $D$ be a minimum total dominating set of $G$. Notice that there exists a minimum total dominating set $D_{i,l}$ of $Q_{i} - N[y_{i}]$ containing $x_{i}$, that is, a vertex from each of $L_{0}(Q_{i}), L_{1}(Q_{i})$ and $L_{2}(Q_{i})$ totally dominates $L_{0}(Q_{i}) \cup L_{1}(Q_{i}) \cup L_{2}(Q_{i}) \cup L_{3}(Q_{i})$; by symmetry, there exists a minimum total dominating set $D_{i,r}$ of $Q_{i} - N[x_{i}]$ containing $y_{i}$, that is, a vertex from each of $R_{0}(Q_{i}), R_{1}(Q_{i})$ and $R_{2}(Q_{i})$ totally dominates $R_{0}(Q_{i}) \cup R_{1}(Q_{i}) \cup R_{2}(Q_{i}) \cup R_{3}(Q_{i})$. For the graph $R$, there exists a minimum total dominating set $D_{0, l}$ of $R - \textsc{Right}$ and a minimum total dominating set $D_{0, r}$ of $R$ containing the \textsc{Right} vertex. For the graph $J$, there exists a minimum total dominating set $D_{n+1, l}$ containing the \textsc{Left} vertex and a minimum total dominating set $D_{n+1, r}$ of $J - \textsc{Left}$.

If $D \cap V(\mathfrak{C}_{n})$ totally dominates $\mathfrak{C}_{n}$, then $|D \cap V(\mathfrak{C}_{n})| \geq 3n + 3$ and $|D| \geq (3n+3) + 3 = 3n+6$. So we may assume that $D \cap V(\mathfrak{C}_{n})$ does not totally dominates $\mathfrak{C}_{n}$. If $x \notin D$, then $D \cap R_{1}(Q_{n}) = \emptyset$ and $D \cap V(\mathfrak{C}_{n})$ totally dominates $\mathfrak{C}_{n} - y_{n}$ and $|D \cap V(\mathfrak{C}_{n})| \geq 3n + 2$, thus $|D| \geq 3n + 2 + 4 = 3n+6$. If $x \in D$, then $D \cap R_{1}(Q_{n}) = \emptyset$ and $D \cap (\mathfrak{C}_{n} - y_{n})$ totally dominates $\mathfrak{C}_{n} - N[y_{n}]$ and $|D| \geq 3n + 2 + 4 = 3n + 6$. There exists a total dominating set with $3n+6$ vertices, such as $D_{0, r} \cup D_{1, r} \cup D_{2, r} \cup \dots D_{n, r} \cup D_{n+1, r}$. Hence, we have that $\gammaup_{t}(R \bullet Q^{(n)} \bullet J) = 3n + 6$.

Let $v$ be an arbitrary vertex. If $v \in R$, then a minimum total dominating set of $R - v$ and $D_{1, r} \cup D_{2, r} \cup \dots D_{n, r} \cup D_{n+1, r}$ form a total dominating set of $G - v$ with $3n+5$ vertices.

If $v \in J$, then $D_{0, r} \cup D_{1, r} \cup \dots D_{n-1, r}$ and a minimum total dominating set of $R_{2}(Q_{n})$ and a minimum total dominating set of $J - v$ form a total dominating set of $G - v$ with $3n+5$ vertices.

If $v \in L_{1}(Q_{1}) \cup L_{2}(Q_{1})$, then there exists two adjacent vertices in $L_{1}(Q_{1}) \cup L_{2}(Q_{1})$ which totally dominates $L_{0}(Q_{1}) \cup L_{1}(Q_{1}) \cup L_{2}(Q_{1}) \cup L_{3}(Q_{1}) - v$, denote this two adjacent vertices by $D^{*}$. Thus $D_{0, l} \cup D^{*} \cup D_{2, l} \cup \dots \cup D_{n, l} \cup D_{n+1, l}$ is a total dominating set of $G - v$ with $3n+5$ vertices.

If $v \in L_{3}(Q_{n}) \cup L_{4}(Q_{n})$, then there exists two adjacent vertices in $L_{3}(Q_{n}) \cup L_{4}(Q_{n})$ which totally dominates $L_{2}(Q_{n}) \cup L_{3}(Q_{n}) \cup L_{4}(Q_{n}) \cup L_{5}(Q_{n}) - v$, denote this two adjacent vertices by $S^{*}$. Thus $D_{0, r} \cup D_{1, r} \cup \dots \cup D_{n-1, r} \cup S^{*} \cup D_{n+1, r}$ is a total dominating set of $G - v$ with $3n+5$ vertices.

Suppose that $v \in L_{0}(Q_{i}) \cup L_{1}(Q_{i}) \cup L_{2}(Q_{i})$ with $i \geq 2$. Thus $D_{0, r} \cup \dots \cup D_{i-2, r}$ and two adjacent vertices in $R_{2}(Q_{i-1})$ and two adjacent vertices in $L_{1}(Q_{i}) \cup L_{2}(Q_{i})$ which totally dominates $L_{0}(Q_{i}) \cup L_{1}(Q_{i}) \cup L_{2}(Q_{i}) \cup L_{3}(Q_{i}) - v$ and $D_{i+1,l} \cup \dots \cup D_{n, l} \cup D_{n+1, l}$ form a total dominating set of $G - v$ with $3n+5$ vertices.

Suppose that $v \in L_{3}(Q_{i}) \cup L_{4}(Q_{i}) \cup L_{5}(Q_{i})$ with $i \leq n-1$. Thus $D_{0, r} \cup D_{1, r} \cup \dots D_{i-1, r}$ and two adjacent vertices in $L_{3}(Q_{i}) \cup L_{4}(Q_{i})$ which totally dominates $L_{2}(Q_{i}) \cup L_{3}(Q_{i}) \cup L_{4}(Q_{i}) \cup L_{5}(Q_{i}) - v$ and two adjacent vertices in $L_{2}(Q_{i+1})$ and $D_{i+2, l} \cup \dots \cup D_{n+1, l}$ form a total dominating set of $G - v$ with $3n+5$ vertices.

Hence, for any vertex $v$ in $V$, we have that $\gammaup_{t}(G - v) \leq 3n+5$, and then $G$ is a $(3n+6)$-$\gammaup_{t}$-vertex-critical graph.
\end{proof}

We can adopt the similar technique to prove that $R \bullet Q^{(n)} \bullet R \bullet R$ is $(3n+7)$-$\gammaup_{t}$-vertex-critical, so we omit the details of the proof.
\begin{theorem}%
The graph $R \bullet Q^{(n)} \bullet R \bullet R$ is a $(3n+7)$-$\gammaup_{t}$-vertex-critical graph with diameter $5n + 9$.
\end{theorem}

\begin{theorem}%
For every integer $k \geq 4$, there are infinitely many graphs that are $k$-$\gammaup_{t}$-vertex-critical with diameter $\left\lfloor \frac{5k - 7}{3} \right\rfloor$.
\end{theorem}
\begin{proof}%
We divide the graphs into four classes according to the value of $k$.
 \begin{enumerate}[label=(\arabic*)]%
\item Suppose that $k \equiv 2 \pmod{3}$ and $k = 3n + 5$. Notice that the graph $A  \bullet  Q^{(n)} \bullet  A$ is a $(3n + 5)$-$\gammaup_{t}$-vertex-critical graph and $\diam(A  \bullet  Q^{(n)} \bullet  A) = 5n + 6 = \frac{5k - 7}{3}$, which has been proved in \cite[Theorem 13]{Goddard2004}.

\item Suppose that $k \equiv 0 \pmod{3}$ and $k = 3n + 6$. Notice that the graph $R  \bullet  Q^{(n)}  \bullet  J$ is a $(3n + 6)$-$\gammaup_{t}$-vertex-critical graph and $\diam(R  \bullet  Q^{(n)}  \bullet  J) = 5n + 7 = \left\lfloor \frac{5k - 7}{3} \right\rfloor$.

\item Suppose that $k \equiv 1 \pmod{3}$ and $k = 3n + 7$. Notice that the graph $R  \bullet  Q^{(n)}  \bullet  R \bullet R$ is a $(3n + 7)$-$\gammaup_{t}$-vertex-critical graph and $\diam(R  \bullet  Q^{(n)}  \bullet  R \bullet R) = 5n + 9 = \left\lfloor \frac{5k - 7}{3} \right\rfloor$.

\item If $k = 4$, then the graph $J$ meet the requirement by \autoref{coalescence}.\qedhere
\end{enumerate}
\end{proof}

\begin{remark}%
As in \cite{MR2917914}, the upper bound in \autoref{IDIAM} is sharp. We provide infinitely many $k$-$i$-vertex-critical graphs with diameter $2(k-1)$ for each $k \geq 2$. For instance, let $B$ be the complete graph on $2t$ vertices with a perfect matching removed, and let $G$ be the graph whose block graph is a path on $k-1$ vertices and every block is a copy of $B$; notice that $i(G) = k$ and $\diam(G) = 2(k-1)$.
\end{remark}
\begin{remark}%
So far, we don't know if the given upper bound on the $k$-$\gammaup_{c}$-vertex-critical graphs is best possible.
\end{remark}

\vskip 3mm \vspace{0.3cm} \noindent{\bf Acknowledgments.} This project was supported by the National Natural Science Foundation of China (11026078). The author would like to express heartfelt thanks to Edwards and MacGillivray for providing us a pdf-file of the paper \cite{MR2917914}. In addition, the author would like to thank the anonymous reviewers for their valuable comments.

\end{document}